\documentclass{article}
\usepackage[utf8]{inputenc}
\usepackage{amsmath, amsthm, amssymb, color}
\usepackage{comment}
\usepackage[margin=1in]{geometry}

\usepackage[style=alphabetic, sortcites]{biblatex}
\renewbibmacro{in:}{}
\DeclareMathOperator{\im}{im}
\newcommand{\F}{\mathbb{F}} 
\newcommand{\Z}{\mathbb{Z}}
\newcommand{\R}{\mathbb{R}}
\renewcommand{\P}{\mathbb{P}}
\newcommand{\Q}{\mathbb{Q}}
\newcommand{\M}{\mathcal{M}}
\newcommand{\C}{\mathbb{C}}
\newcommand{\Grad}{\nabla}
\newcommand{\tensor}{\otimes}
\DeclareMathOperator{\codim}{codim}

\newtheorem{theorem}{Theorem}
\newtheorem{proposition}{Proposition}
\newtheorem{definition}{Definition}
\newtheorem{lemma}{Lemma}
\newtheorem{remark}{Remark}
\newtheorem*{conjecture}{Conjecture}

\title{Rational Quantum Cohomology of Steenrod Uniruled Manifolds}
\author{Semon Rezchikov}
\date{October 2021}

\begin{document}

\maketitle

\begin{abstract}
    We show that if a semipositive symplectic manifold $M^{2n}$ is \emph{Steenrod uniruled}, in the sense that the quantum Steenrod power of the point class does not agree with its classical Steenrod power for infinitely many primes, then the (rational) quantum product on $M$ is deformed. This bridges the gap between the recent advances towards the Chance-McDuff conjecture utilizing quantum Steenrod operations, and the natural formulation of the Chance-McDuff conjecture in terms of rational Gromov-Witten theory. 
\end{abstract}

\section{Introduction}
This paper shows that recent results towards the Chance-McDuff conjecture, which use equivariant Floer theory to establish results about invariants of semipositive symplectic manifolds taking values in fields of positive characteristic, imply the original formulation of the Chance-McDuff conjecture in terms of rational Gromov-Witten invariants. 

Let $(M, \omega)$ be a semipositive symplectic manifold, and let $\phi: M \to M$ be a Hamiltonian diffeomorphism with finitely many periodic points. A natural way to produce such a $\phi$ is via the flow of a Hamiltonian $H: M \to \R$ which is the moment map of an $S^1$-action on $M$; in this case, the periodic points for an irrational rotation are just the fixed points, which are the critical points of the Hamiltonian. McDuff \cite{McDuff-uniruled} proved Hamiltonian $S^1$-manifolds are \emph{strongly uniruled} in the sense that there is a nonzero 
Gromov-Witten $\langle pt, a_2, a_3 \rangle_A$ with $A \neq 0$. 

Let $\F$ be a field, and let the Novikov ring $\Lambda_\F$ consist of the series
\begin{equation}
    \sum_A c_A q^A \text{ such that } c_A \in \F, A \in im(\pi_2(M) \to H_2(M, \Z)) , 
\end{equation}
where for each term either $A = 0$ or $\int_\omega A > 0$ and such that for any real $C$, $\#\{c_A \neq 0 | \int_\omega A < C\}< \infty$. We associate the grading $|q^A|= 2c_1(A)$, although of course $\Lambda_\F$ is only a filtered ring. Then McDuff's result implies that the small quantum cohomology $QH^*(M, \Lambda_\Q)$ is \emph{not} isomorphic to $H^*(M, \Lambda_\Q)$ as a ring.

Since a Hamiltonian $H$ generating an $S^1$-action must be a perfect Morse function, the Floer chain complex $CF^*(H, \Lambda_\F)$ of $H$ has trivial differential, and is thus isomorphic to $H^*(M, \Lambda_\F)$ as a graded abelian group. This property is axiomatized via the notion of a \emph{pseudorotation}.

\begin{definition}
Let $\text{char } \F = 0$ if $M$ is not semipositive, and otherwise let $\F$ be an arbitrary field. We say that a Hamiltonian diffeomorphism $\phi: M \to M$ is a \emph{pseudorotation} (see e.g. \cite{S-PRQS}) if the set of periodic points of $\phi$ is finite and coincides with the set of fixed points, if there is identity of dimensions
\[ \sum_{x \in M: \phi(x) = x} \dim_{\Lambda_F} HF_{loc}(x, \phi) = \dim_{\Lambda_F} H^*(M, \Lambda_F) \]
where $HF_{loc}$ is the \emph{local Floer homology} of a fixed point of $\phi$ (\cite{Floer3}, \cite{ginzburg-conley}), and ever term of the sum above is strictly positive. 
\end{definition}

A basic expectation about the symplectic topology of manifolds admitting Hamiltonian diffeomorphisms with finitely many periodic points is the following
\begin{conjecture}[Chance-McDuff]
Suppose $\phi$ has finitely many periodic points. Then $QH^*(M, \Lambda_\Q) \neq H^*(M, \Lambda_\Q)$ as a ring. 
\end{conjecture}

Recent partial progress towards this conjecture been restricted to the case when $\phi$ is a pseudorotation. Using one method \cite{CGG}, given additional index-theoretic conditions on the periodic points of $\phi$, one can utilize combinatorial arguments about index divisibility and iteration of the quantum product in local Floer cohomology to conclude the existence of nontrivial rational Gromov-Witten invariatns of $M$. Alternatively, several works have given conceptual arguments using \emph{equivariant localization in Floer homology} \cite{Seidel-pants} \cite{SZhao-pants} together with \emph{quantum Steenrod operations} \cite{Fukaya-Steenrod}, \cite{Wilkins} to argue the quantum Steenrod operations cannot agree with their classical counterparts.

In particular, \cite{S-PRQS} and subsequently \cite{S-PRQSR}, \cite{CGG2} establish that if $M$ is a monotone symplectic manifold admitting an $\F_2$-pseudorotation, then $M$ is \emph{$\F_2$-uniruled}, a notion introduced in \cite{S-PRQS}. We review this notion in Section \ref{sec:quantum-steenrod-review}; being $\F$-uniruled implies that for any compatible almost complex structure on $M$, there is a pseudoholomorphic sphere passing through every point of $M$.  Moreover, an upcoming paper \cite{SSW} will establish that if $M$ is a semipositive symplectic manifold admitting an $\F_p$-pseudorotation then $M$ is $\F_p$-uniruled.

The main theorem of this paper is that such results using equivariant Floer homology are sufficient to establish the Chance-McDuff conjecture formulated using rational Gromov-Witten theory. 

\begin{theorem}
\label{thm:main-theorem}
Let $(M, \omega)$ be a  symplectic manifold of dimension $2n$ which is either monotone or has minimal Chern number $N>1$, and is  $\F_p$-uniruled for infinitely many primes $p$. Then the small rational quantum product on $M$ is deformed, i.e $QH^*(M, \Lambda_\Q) \neq H^*(M, \Lambda_\Q)$ as a ring. If $M$ is not monotone and has $N=1$, then the same conclusion holds when $M$ is $\F_2$-uniruled.
\end{theorem}
\begin{remark}
\label{rk:p-constraint-1}
The proof only uses a single prime $p$, and explicit descriptions of valid $p$ may be obtained from the class $[\omega]$ and the cohomological invariants of $M$. See Lemma \ref{lemma:making-omega-integral} for the main obstruction, which has to do with making sure one can reduce $[\omega]$ mod $p$.
\end{remark}

The proof uses certain structural properties of the quantum Steenrod operations elaborated in \cite{Wilkins}, \cite{Wilkins-PSS}, \cite{Seidel-formal}, and \cite{SeidelWilkins}, which review in Section \ref{sec:quantum-steenrod-review}, along with more basic facts about quantum cohomology. Section \ref{sec:steenrod-coefficients-must-be-calabi-yau} shows that if $QH^*(M, \Lambda_\Q) = H^*(M, \Lambda_\Q)$ as rings, then in the total quantum Steenrod power of the fundamental class, all coefficients of $q^A$ must be zero unless $c_1(A) < 1$. Section \ref{sec:gw-theory-review} reviews the moduli spaces defining rational and equivariant Gromov-Witten invariants.  Section \ref{prop:calabi-yau-classes-cannot-unirule} shows that if $c_1(A) \leq 0$, then the pseudoholomorphic spheres in the class $A$ cannot pass through a generic point of $M$. Combining these propositions, we conclude the proof of the main theorem below. 

\begin{proof}[Proof of Theorem \ref{thm:main-theorem}]
Suppose that $QH^*(M, \Lambda_\Q) = H^*(M, \Lambda_\Q)$. 
By Proposition \ref{prop:steenrod-coefficients-are-calabi-yau} 
the there is a prime $p$ such that the non-classical terms in $Q\Sigma_{[M]}(1)$ are $f_A q^A$ for $c_1(A) < 1$; but by Proposition \ref{prop:calabi-yau-classes-cannot-unirule}, $f_A = 0$. Thus $M$ is not $\F_p$ uniruled for this $p$. We have proven the contrapositive. 
\end{proof}

\textbf{Acknowledgements.} I thank Egor Shelukhin for his interest in this work, for pointing out an error in an initial version of the argument presented in this paper, and for his comments on this manuscript. I also thank Denis Auroux for providing feedback regarding dimension-counting arguments in the style of Section 5. Additional acknowledgements to be added later.

\section{Review of Quantum Steenrod Operations}
\label{sec:quantum-steenrod-review}
On a semipositive sympletic manifold $M$, the $3$-point GW invariants \cite{McDuffSalamon-BIG} are multilinear maps  
\begin{equation}\langle \cdot, \cdot, \cdot \rangle_A: H^{*}(M, \Z)^{\tensor 3} \to \Z;
\end{equation}
by multilinearity they factor through $(H^{*}(M, \Z)/Tors)^{\tensor 3}$. We can then define 
\begin{equation}
    \langle \cdot, \cdot, \cdot \rangle_A: H^{*}(M, \F)^{\tensor 3} \to \F;
\end{equation}
by setting $\langle x, y, z \rangle_A =0$ whenever one of $(x,y,z)$ does not land in $Im(H^{*}(M, \Z)/Tors \xrightarrow{\tensor \F} H^*(M, \F))$, and otherwise by lifting $(x,y,z)$ to integral cohomology and applying the homomorphism $\Z \to \F$ to the resulting Gromov-Witten invariant.

These invariants give maps
\begin{equation}
\begin{gathered}
    *_A:  H^*(M, \F)^{\tensor 2} \to H^{*-2c_1(A)}(M, \F) \\ 
    \int_M x *_A y z = \langle x, y, z \rangle_A
\end{gathered}
\end{equation}
which when extended $\Lambda_\F$-linearly and combined into 
\begin{equation}
    x * y = \sum_{A} (x *_A y) q^A
\end{equation}
give the $\Lambda_\F$-module $H^*(M, \Lambda_\F)$ the structure of a commutative $\Lambda_\F$-algebra -- the \emph{small quantum cohomology} $QH^*(M, \Lambda_\F)$ of $M$. 

The discussion above establishes the following

\begin{lemma}
If $\langle x, y, z \rangle_A = 0$ for all $A \neq 0$ and $\F=\Q$, then this holds for any other field $\F$.
\end{lemma}
\begin{proof}
If one of $\{x,y,z\}$ will does not lift to integral cohomology then the invariant is zero by definition. If they lift, then the invariant is zero any of the lifts are torsion classes. If none of the lifts are torsion classes, then the invariant agrees with the rational Gromov-Witten invariant associated to the image of the lifts in rational cohomology. 
\end{proof}

One defines the small quantum connection on $H^*(M, \Lambda)[[t]]$ via 
\begin{equation}
    \Grad_a \gamma = t \partial_a \gamma + a * \gamma. 
\end{equation}
where $*$ is extended $t$-linearly and $|t|=2$. Here $a \in H^2(M, \Z)$ and $\partial_a q^A = (a \cdot A) q^A$.

The Quantum Steenrod Operations are certain maps depending on a prime $p$
\begin{equation}
    Q\Sigma_b: H^*(M, \F_p) \to (H^*(M, \Lambda) [[t,\theta]])^{*+p|b|}
\end{equation}
where $b \in H^*(M, \F_p)$ and $|t|=2$, $|\theta|=1$, which $q$-deform the usual operations in the sense that their reduction modulo the maximal ideal of $\Lambda_\F$ agrees with the operation 
\begin{equation}
    H^*(M, \F_p) \ni x \mapsto St_p(b)x
\end{equation}
where $St_p$ is the total Steenrod power 
\begin{equation}
    St_p: H^{|b|}(M, \F_p) \to (H^*(M, \F_p) [[t,\theta]])^{p|b|}. 
\end{equation}
We use the conventions of \cite{Seidel-formal} for $St_p$; up to explicit signs and nonzero constants, the coefficients of the terms of $St_p$ are the usual Steenrod powers $P^i$ or their Bocksteins $\beta P^i$ (or alternatively the Steenrod squares $Sq^i$ if $p=2$) defined as in \cite{hatcher}.

\begin{definition}[\cite{S-PRQS}]
We say that $M$ is $\F_p$-uniruled if 
\begin{equation}
Q\Sigma_{[M]}(1) \neq St_p([M])(1) = \pm [M]t^{(p-1)n}. 
\end{equation}
where $[M]$ is the cohomology class Poincare-dual to a point (the ``point class''). See \cite[Eq. 10.9]{Seidel-formal} for a formula for the sign.
\end{definition}
This implies that for every compatible almost complex structure on $M$ there is is a pseudoholomorphic sphere through every point of $M$; see Remark \ref{remark:compactness-uniruled}. 

The classical and quantum Cartan relations are the respective identities
\begin{equation}
    St_p(a)St_p(b) = (-1)^{c(a, b)}St_p(ab), \; Q\Sigma_b Q\Sigma_c = (-1)^{c(a, b)}  Q\Sigma_{b * c}, \text{ where } c(a, b) = |a||b|\frac{p(p-1)}{2},
\end{equation}
where we extend all of the above maps to endomorphisms of $H^*(M, \Lambda) [[t,\theta]]$ by $t, \lambda, q^A$-linearity, and we extend  
the lower argument of $Q\Sigma$ to take $\Lambda_\F$-linear sums by requiring that \begin{equation}
    Q\Sigma_{b}= \sum_A q^{pA} Q\Sigma_{b_A}\text{ for }b = \sum_A b_A a^A.  
\end{equation}

Recently, Seidel and Wilkins \cite{SeidelWilkins} have established that \emph{the quantum Steenrod operations are covariantly constant with respect to the quantum connection}, namely that we have an identity
\begin{equation}
\label{eq:covariant-constancy}
    Q\Sigma_b \Grad_a = \Grad_a Q \Sigma_b.
\end{equation}

Let $I_{diff, p} \subset \Lambda_{\F_p}$ be  the ideal generated by those elements $q^A$ such that $\partial_a q^A$ for all $a \in H^2(M, \Z)$. Equation (\ref{eq:covariant-constancy}) immediately implies 
\begin{lemma}
\label{lemma:basic-restriction}
Suppose that $\langle x, y, z \rangle_A = 0$ for all $A \neq 0$. Then $Q\Sigma_b c= St(b)c \mod I_{diff}$. 
\end{lemma}
\begin{proof}
The assumption implies that $\Grad_a = t\partial_a$. Since $t$ is not a zero-divisor (\ref{eq:covariant-constancy}) implies that for any $x \in H^*(M, \F_p)$, 
\begin{equation}
    \partial_a Q\Sigma_b(x) = Q\Sigma_b(\partial_a x) = 0. 
\end{equation}
\end{proof}

\section{Steenrod Coefficients must be Calabi-Yau}
\label{sec:steenrod-coefficients-must-be-calabi-yau}

In this section, we prove
\begin{proposition}
\label{prop:steenrod-coefficients-are-calabi-yau}
Let $(M, \omega)$ be semipositive, and assume that the small quantum cohomology of $M$ is classical. Then, if $M$ is monotone or has minimal chern number $N>1$, then there are infinitely many primes $p$ such that if $q^A$ has a nonzero-coefficient in $Q\Sigma_{[M]}(1)$ for this prime, then $c_1(A) \leq 0$. If $M$ is not monotone and has $N=1$, then this conclusion holds for $p=2$.
\end{proposition}

We recall that the semipositivity condition (also refered to as \emph{weak monotonicity} \cite{HS-Novikov}) means \cite[Exercise 6.4.3]{McDuffSalamon-BIG} that either
\begin{enumerate}
    \item $M$ is monotone: $c_1(M)(A) = \kappa [\omega](A)$ for some $\kappa > 0$, for every $A \in \pi_2(M)$;
    \item $M$ is spherically Calabi-Yau, i.e. $c_1(M)(A) = 0$ for every $A \in \pi_2(M)$; or
    \item $M$ has minimal Chern number $N \geq n-2$. 
\end{enumerate}

The proof will require several cases, and in the worst case ($n=3$, $N=1$) we will use an analytical result proven in the final section. Before we proceed, we discuss how to make $[\omega]$ into an integral class.

\begin{lemma}
\label{lemma:making-omega-integral}
If $(M^{2n}, \omega)$ is a semipositive symplectic manifold then there is a constant $P>0$ such that for all $p > P$, $\omega$ can be deformed through semipositive symplectic forms into a symplectic form $\tilde{\omega}$ such that $[\tilde{\omega}]$ is integral and such that $[\tilde{\omega}]^n \neq 0 \mod p$. 
\end{lemma}
\begin{proof}
Suppose $M$ is monotone or spherically Calabi-Yau. Write $H^2_S(M) = Im(\pi_2(M) \to H_2(M, \Z)/Tor)$. Write $H^2_S(M) = Hom(H^2_S(M), \Z)$, and $H^2_S(M)^\perp = \ker(Hom(H_2(M, \Z)/Tors, \Z) \to H^2_S(M)$. Then $H^2(M, \R) = H^2_S(M) \tensor \R \oplus H^2_S(M)^\perp \tensor \R$. Under this decomposition write $[\omega] = [\omega]_S + [\omega]_S^\perp$. 

If $M$ is monotone or spherically Calabi-Yau, since $c_1(M)$ is integral we can let $\tilde{\omega}_1$ to be a rescaling of $\omega$ so that $[\tilde{\omega}_1]_S$ an integral class. In general, the set of cohomology classes of symplectic forms deformation equivalent to $\omega$ projects to a subset of $H^2(M, \R)$ which contains an open subset $U$ of $[\omega]$. The image of $H^2_S(M)^\perp$ in $H^2_S(M)^\perp \tensor \R$ is dense; thus, we can perturb $\tilde{\omega}_1$ among semi-positive symplectic forms with cohomology classes in $U \cap ([\tilde{\omega}_1]_S + H^2_S(M)^\perp \tensor \R)$ to a symplectic form $\tilde{\omega}$ such that $[\tilde{\omega}_1]_S = [\tilde{\omega}_S]$, while $[\tilde{\omega}]^\perp_S$ is rational. Thus $[\tilde{\omega}]$ is rational. Rescaling $\tilde{\omega}$ by an integer we can assume that $\tilde{\omega}$ is integral. On the other hand, if $N \geq n-2$ then then any other symplectic form on $M$ is semipositive; therefore, we can also deform $\omega$ to a rational symplectic form through semipositive symplectic forms by a small perturbation, and subsequently rescale to get an integral symplectic form.

Write $[\tilde{\omega}]^n = K [M]$; since $\tilde{\omega}$ is symplectic, $K \neq 0$. Then $P=K$ will do. 
\end{proof}
\begin{remark}
Continuing on Remark \ref{rk:p-constraint-1}, the factors of $p$ are allowed are a problem in the approximation of $\omega$ by rational classes, and in concrete cases, one may follow the reasoning of the lemma above to derive explicit primes $p$ for which the conclusion of the lemma is valid. In the remainder of the proof we will have no further conditions on $p$, unless $N=1$ and $n=3$, where we will change our strategy and take $p=2$.
\end{remark}

\begin{remark}
\label{rk:deformation-invariance}
The rational Gromov-Witten invariants of a semipositive symplectic manifold are independent of symplectic form up to deformation through semipositive symplectic forms \cite{McDuffSalamon-BIG}. This holds because one can use almost complex structures \emph{tamed} by $\omega$ to define the Gromov-Witten moduli spaces, and because the set of almost complex structures tamed by $\omega$ is open in the set of almost complex structures and correspondingly the set of symplectic forms taming a fixed almost complex structure is open in the space of symplectic forms. The quantum Steenrod powers as defined in \cite{SeidelWilkins} require one to choose $\omega$-compatible almost complex structures; but for the Gromov compactness results used in the paper, $\omega$-tame almost complex structures would suffice. Thus, one can show that the deformations of Lemma \ref{lemma:making-omega-integral} preserve the property of the underlying symplectic manifold being Steenrod uniruled for any given prime $p$.
\end{remark}

\begin{proof}[Proof of Proposition \ref{prop:steenrod-coefficients-are-calabi-yau}]

If $c_1(M) = 0$ the conclusion follows. 

Otherwise we will prove the contrapositive; thus, assume that $\langle x, y, z\rangle_A = 0$ for all $A \neq 0$. Then quantum multiplication agrees with the classical cup product.

First, suppose $\omega$ is integral. Note that (reducing $[\omega]$ mod $p$),  $St_p([\omega]) = G_p([\omega])$ for $G_p(z) \in \F_p[z]$ an explicit polynomial. Indeed, there is a map $f: M \to \C P^\infty$ such that $f^*c_1 = [\omega]$, where $H^*(\C P^\infty, \F_p) = \F_p[c_1]$. The claim follows because the Steenrod operations are natural.

We will repeatedly utilize the following 
\begin{lemma}
\label{lemma:fundamental-constraint}
Let $a \in H^2(M, \F)$; then if $p>n$ we can write 
\begin{equation}
    Q\Sigma_a(b) = St_p(a)b + \sum_{i} f_i a^{A_i} + \sum_j g_j q^{B_j}
\end{equation}
where $c_1(A_i) = 0$, $c_1(B_j) = p$, $A_i \neq 0$, and $f_i, g_j \in (t, \theta) H^*(M, \F)[[t, \theta]]$. If $b=1$ or $N>1$ then $g_j =0$ for each $j$.  
\end{lemma}
\begin{proof}
We have that $Q\Sigma_{a}b= St(a)b + f$
for $f \in I_{diff, p}$ of degree $2p$. Write $f = \sum_i f_i q^{A_i} + \ldots $ for $\partial_a q^{A_i} = (a \cdot A_i) a^{A_i} = 0$, $A_i \neq 0$, and  $f_i \in H^*(M, \F)[[t, \theta]]$. Write $A_i = pB_i + T_i$ where $T_i$ is a torsion class, and suppose $f_i \neq 0$.

We see that $|f_iq^{A_i}| = |f_i|+2 c_1(pB_i+T_i) =|f_i| + 2pc_1(B_i) = |b|+2p$. Dividing by $2p$ and using that $p>n$ we see that $c_1(B_i) \leq 1$. So if $N>1$ we have that $c_1(A_i) \leq 0$.

By Remark \ref{rk:no-negative-Chern-numbers} we actually have that if $c_1(A_i) \geq 0$ for any $i$. So we have established the theorem when $N>1$.

Moreover, writing $f_i \in f^0_i + (t, \theta)$ we have
that $f^0_i=0$ since 
\begin{equation}
\sum_i f^0_i a^{A_i}= a^{*n}*b - a^nb = 0
\end{equation}
since again, the quantum product is classical. By degree considerations this proves the last sentence when $b=1$.

\end{proof}

In the case where $M$ is monotone or when $N>1$, we apply Lemma \ref{lemma:making-omega-integral}, and replace $\omega$ with $\tilde{\omega}$ as in the lemma. By Remark \ref{rk:deformation-invariance}, this preserves the rational Gromov-Witten invariants of $M$, as well as the property of being $\F_p$-Steenrod uniruled.   We then have that $[\omega]^{* n} = [\omega]^n = K[M]$. Note that $Q\Sigma_{[M]}(1) = (1/K)Q\Sigma_{[\omega]^{*n}}(1)$, which makes sense since $0 \neq K \in \F_p$ is invertible.

We have by Lemma \ref{lemma:fundamental-constraint} that
\begin{equation}
    Q\Sigma_{[\omega]}(1) = St_p([\omega]) + \sum_i f_i q^{A_i} = G_p([\omega]) + \sum_i f_i q^{A_i}, \; c_1(A_i) = 0, A_i \neq 0. 
\end{equation}

This is the base case of an induction on $k$ establishing that 
\begin{equation}
\label{eq:inductive-hypothesis}
    Q\Sigma_{[\omega]^k}(1) = Q\Sigma_{[\omega]}^k(1) = St([\omega]^k)+ \sum_i f_i q^{A_i},\; c_i(A_i) = 0, A_i \neq 0. 
\end{equation}

In the case that $M$ is monotone we further have that $f_i = 0$ in the base case and the inductive hypothesis. 

Now, for any $a \in H^2(M, \Z)$, there is an operation \cite{SeidelWilkins}
\begin{equation}
    \Pi_{a, b}: H^*(M, \F_p) \to (H^*(M)[[t, \theta]])^{*+|a|-2+p|b|}
\end{equation}
satisfying
\begin{equation}
\label{eq:pi-relation-1}
    t \Pi_{a, b}(c) = Q\Sigma_b(a * c) - a* Q\Sigma_b(c).
\end{equation}
Moreover, it also satisfies
\begin{equation}
\label{eq:pi-relation-2}
    \Pi_{a, b}(c) = \partial_a Q\Sigma_b(c) = 0.
\end{equation}
Since $t$ is not a zero-divisor we conclude that $Q\Sigma_b(\omega^k) = \omega^k Q \Sigma_b(1)$ for any $k \geq 0$.

Thus assuming (\ref{eq:inductive-hypothesis}) we have that
\begin{equation}
    Q\Sigma_{[\omega]} Q\Sigma_{[\omega]}^k(1) = Q\Sigma_{[\omega]}( G_p({[\omega]})^k) + \sum_i q^{A_i}Q\Sigma_{[\omega]}(f_i) =
    G_p({[\omega]})^kQ\Sigma_{{[\omega]}}(1) + \sum_i q^{A_i}Q\Sigma_{[\omega]}(f_i).
\end{equation}

In the case that $N > 1$ we see that $g_i = 0$ in Lemma \ref{lemma:fundamental-constraint}; applying that lemma to the terms $Q\Sigma_{[\omega]}(f_i)$ in formula above we see that the inductive hypothesis is preserved. On the other hand, if we are in the monotone case the $f_i$ were zero and thus applying the formula from the base case we see that the inductive hypothesis has been preserved as well. This establishes the induction and for $k=n$ proves the theorem when $M$ is monotone or when $N>1$.

The remaining case is $N=1$ and $M$ is not monotone, which implies that $2 \leq n \leq 3$.

Suppose $n =2$. Then by Poincare duality we can find a pair of classes $a, b \in H^2(M, \F_p)$ such that $ab = [M]$. Choosing  $p>n=2$, we use quantum Cartan to write $Q\Sigma_{[M]}(1) = Q\Sigma_bQ\Sigma_{a}(1)$ and conclude the proposition by applying Lemma \ref{lemma:fundamental-constraint} twice. 

The final and trickest case is if $n=3$. In this case we will specialize to $p=2$.  By Poincare duality we can find a pair of classes $a\in H^4(M,\Z)$, $b\in H^2(M, \Z)$ such that $ab=[M]$ integrally; and thus the same holds for the mod-$2$ reductions of these classes.

We will decompose $Q\Sigma_{[M]}(1) = Q\Sigma_bQ\Sigma_a(1)$. Let us study the terms in 
\begin{equation}
    Q\Sigma_a(1) = St_2(a) + \sum_i f_iq^{A_i}, A_i \neq 0. 
\end{equation}
Then $St_2(a) = t^2Sq_0(a) + t\theta Sq_1(a) + Sq_2(a).$  Since $a$ is integral, $Sq_1(a)=0$ since $Sq_1$ is the Bockstein. For the remaining terms we have $|f_iq^{A_i}| = 4$, so arguing as in Lemma \ref{lemma:fundamental-constraint} we write $A_i = pB_i + T_i$ where $T_i$ is torsion, and conclude that $|f_i| + 4c_1(B_i) = 8$. We write $f_i = \theta^{x_i} t^{y_i} f'_i$ with $f'_i \in H^*(M, \F_2)$. Then $|x_i| + |y_i| > 0$, and in particular $|f_i|>0$. Thus $c_1(B_i) \leq 1$. If $c_1(B_i) = 1$ then $|f'_i|\leq 3$. We will establish by an analytical argument in Section \ref{sec:uniruling-cohomology-classes} that
\begin{proposition}
\label{prop:bad-case}
In the situation described above, $f_i=0$ for each $i$.
\end{proposition}
Given this proposition we conclude by writing $Sq_2(a) = r[M] = rba$, and utilizing (\ref{eq:pi-relation-1}), (\ref{eq:pi-relation-2}) to compute that
\begin{equation}
    Q\Sigma_{[M]}(1) = Q\Sigma_b((t^2+rb)a) = (t^2+rb)Q\Sigma_ba = (t^2+rb)\left(Sq_2(b)a + \sum_i f_i q^{2B_i + T_i}\right)
\end{equation}
where $B_i, T_i \in H_2(M, \Z)$ with $B_i$ torsion-free and $T_i$ torsion. By degree considerations as above we conclude that $c_1(B_i) \leq 1$, and if $c_1(B_i) = 1$ then $f_i = \theta^{x_i}t^{y^i} f'_i$ with $f'_i \in H^*(M, \F_2)$ and $|f'_i| \leq 3$. We then invoke one more proposition proven in Section \ref{sec:uniruling-cohomology-classes}:
\begin{proposition}
\label{prop:bad-case-2}
In the situation above, if $Q\Sigma_{[M]}(1) = \sum_i \theta^{x_i}t^{y_i}f'_i q^{A_i}$ with $c_1(A_i) = 2$ then $|f'_i| > 3$. 
\end{proposition}
Thus $c_1(2B_i + T_i) \leq 0$ and we have proven the proposition. 

\end{proof}


\section{Rational and Equivariant Gromov Witten Invariants}
\label{sec:gw-theory-review}
In this section we review the definitions of the moduli spaces giving rise to the usual Gromov-Witten invariants as well as the quantum Steenrod powers.

Let $f$ be a Morse function with critical point set $Crit(f)$, and let $g$ be a metric such that the gradient flow of $f$ is Morse-Smale.  If $x$ is a critical point of $f$ with Morse index $|x|$, then in our conventions the stable manifold $W^s(x)$ has dimension $|x|$, while the unstable manifold $W^u(x)$ has dimension $2n-|x|$. Write $\F_x$ for the $1$-dimensional $\F$-vector space generated by the two orientations of $W^s(x)$ modulo the relation that the sum of the orientations is zero. The Morse complex is 
\begin{equation}
    CM^k(f) = \bigoplus_{|x|=k}\F_x
\end{equation}
and it has a cohomological differential such that its cohomology computes $H^*(M, \F)$. 

Write $\mathcal{J}(M)$ or for the set of compatible almost complex structures on $M$ and write $\mathcal{J}(S^2, M)$ for the set of smooth maps from $S^2$ to $\mathcal{J}(M)$. 

\paragraph{Rational GW theory.} 
Given $A \in \im(\pi_2(M) \to H_2(M, \Z))$ and an element $J_z \in \mathcal{J}(S^2, M)$ with $z \in S^2$, we have the moduli space
\begin{equation}
    \M_{0,3}(M, A, J_z) = \{u: S^2 \to M; \bar\partial_{J_z}u = 0\}.
\end{equation}
This is equipped with maps 
\begin{equation}
    ev_p: \M_{0,3} \to M \text{ for each }p \in S^2
\end{equation}
where we will choose $p\in\{0,1,\infty\}$. Write $Aut(\P^1, \{S\})$ for the group of automorphisms of $\P^1$ which fix $S \subset \P^1$. If $J_z$ is independent of $z$ (i.e., $J_z \in \mathcal{J}(M)$) then we also have moduli spaces $\overline{\M}_{0,2}$ and $\overline{\M}_{0,1}$ defined by quotienting 
$\overline{\M}_{0,3}$ by $Aut(\P^1, \{\infty\})$ and $Aut(\P^1, , \{\infty, 1\})$, respectively. 

Each of the above moduli spaces is contained in its Gromov-Kontsevich compactification, and contains the subspace of simple maps:
\begin{equation}
\M^*_{0, k}(M, A, J_z) \subset \M_{0, k}(M, A, J_z) \subset \overline{\M}_{0, k}(M, A, J_z)
\end{equation}
The methods of \cite{McDuffSalamon-BIG} can be used to establish 
\begin{proposition}
There exist $J_z$ such that
\[ ev_0 \times ev_1 \times ev_\infty: \M^*_{0, 3}(M, A, J_z) \to M \times M \times M\]
is a pseudocycle \cite{McDuffSalamon-BIG} of dimension $2n+2c_1(A)$.
\end{proposition}

The rational Gromov Witten invariant $\langle a, b, c\rangle_A$ is defined by defining the map on generators of $CM^*(f)$ via
\begin{equation}
\label{eq:three-pointed-gw-invt-definition}
    \langle x, y, z\rangle_A = \#  \M^*_{0, 3}(M, A, J_z) \times_{M \times M \times M} W^u(x) \times W^u(y) \times W^u(z).
\end{equation}
The cardinality of the set on the right hand side is finite whenever $|x|+|y|+|z|-2c_1(A) = 2n$, and the points are counted with sign according to the canonical orientation on the fiber product; otherwise the count is defined to be zero.

\paragraph{Equivariant GW theory.}
The terms the quantum Steenrod operation $Q\Sigma_a$ can also be described in a similar way \cite{SeidelWilkins}. Let 
\begin{equation}
    S^\infty = \{(w_0, w_1, \ldots) | w_i \in \C, w_i = 0 \text{ for } i >>0\} = \bigcup_{k \geq 0} S^k. 
\end{equation}
There is an action of $\Z/p$ on $S^\infty$ where the generator $\sigma$ acts by multiplication by $\zeta = e^{2\pi i/p}$.
There are smooth manifolds with corners $\Delta_i \subset S^\infty$ homeomorphic to disks; the $\Delta_i$ together with their images under the $\Z/p$ action give give $S^\infty$ the structure of a free $\Z/p$-CW complex. 

There is an action of $\Z/p$ on $\P^1$ generated multiplication by $\zeta$ and we write 
\begin{equation}
    C = \P^1, z_{C, 0} = 0, z_{C, 1} = \zeta^{1/2}, z_{C, 2} = \zeta^{3/2}, \ldots, z_{C, p} = \zeta^{-1/2}, z_{C, \infty} = \infty.
\end{equation}

Given an element $J_z \in \mathcal{J}(S^2, M)$, write $Hom(TS^2, TM)$ for the space of linear maps maps from $TS^2$ to $TM$ when each bundle is pulled back to $S^2 \times M$. Note that there is an action of $\sigma$ on $\mathcal{J}(S^2, M)$ and on $Hom(TS^2, TM)$ by pullback.

The operation $Q\Sigma_a$ is defined by choosing perturbation data 
\begin{equation}
\label{eq:quantum-steenrod-perturbation-data}
    \begin{gathered}
    \{J_{z, w}\}_{w \in S^\infty} \in \mathcal{J}(S^2, M) , \{\eta_{z, w}\}_{w \in S^\infty} \in Hom(TS^2, TM),  \\
    J_{\sigma(z), w} = J_{z, \sigma(w)},\; \eta_{\sigma(z), w} = \eta_{z, \sigma(w)}, \;\eta_{z, w}i = -J_z\eta_{z, w}. 
    \end{gathered}
\end{equation} which vary smoothly on each $\Delta_i$, and defining for each $A \in H_2(M, \Z)$, $x_0, \ldots, x_p, x_\infty \in Crit(f)$, the moduli space
\begin{equation}
\label{eq:equivariant-moduli-space}
\M_{A}(\Delta_i \times C, x_0, \ldots, x_p, x_\infty) = \left\{ 
\begin{array}{c|l}
w \in \Delta_i \setminus \partial \Delta_i, &  \bar{\partial}_{J_{w, z}}u = \eta_{w, z}, u_*[C] = A \\ 
u: C \to M & u(z_{C, i}) \in W^u(x_i), i = 0, \ldots, p; u(z_{C, \infty}) \in W^s(x_\infty)
\end{array}\right\}. 
\end{equation}

For generic choices of perturbation data this moduli space has dimension
\begin{equation}
    i + 2c_1(A) + |x_\infty| - \sum_{i=0}^p|x_i|. 
\end{equation}
We then define 
\begin{equation}
    \Sigma_A^i: CM^*(f)\tensor CM^*(f)^{\tensor p} \to CM^{* - i-c_1(A)}(f)
\end{equation}
by using (\ref{eq:equivariant-moduli-space}) to define this map's coefficients. 
This descends to a map on homology, and finally if $a, b \in H^*(M, \F)$, we have
\begin{equation}
    Q\Sigma_a(b) = (-1)^{|a||b|}\sum_A q^A \sum_k \Sigma_A^{2k}(a, b, \ldots b) t^k + (-1)^{|a|+|b|}\Sigma_A^{2k+1}(a, b, \ldots b)t^k\theta. 
\end{equation}

\section{Uniruling Cohomology Classes}
\label{sec:uniruling-cohomology-classes}
In this section we first prove an equivariant symplectic incarnation of a well known fact in algebraic geometry, that a projective variety cannot be uniruled by spheres of low Chern number:

\begin{proposition}
\label{prop:calabi-yau-classes-cannot-unirule}
Write 
\begin{equation}
    Q \Sigma_{[M]}(1) = f_0 + \sum_{A \neq 0} f_A q^A. 
\end{equation}
If $f_A \neq 0$ then $c_1(A) > 1$.
\end{proposition}
\begin{proof}
To define Morse cohomology let us use a Morse function with a unique minimum $p_-$ and a unique maximum $p^+$. 
We choose a sequence of perturbation data $J^i_{z, w}, \eta^i_{z, w}$ as in (\ref{eq:quantum-steenrod-perturbation-data}) which limit to $J^\infty_{z, w} = J^\infty \in \mathcal{J}(M), \eta^\infty_{z, w} = 0$; If $f_A$ is nonzero then the moduli spaces (\ref{eq:equivariant-moduli-space}) defining the corresponding count are nonempty for all $i$. Thus, the Gromov compactification of the corresponding moduli space for the perturbation data $(J^\infty_{z, w}, 0)$ is also nonempty. Let 
\begin{equation}
    \mathcal{M}_A(C, x_0, \ldots, x_p; x_\infty) = \left\{ 
\begin{array}{c|l}
u: C \to M &  \bar{\partial}_{J^\infty}u =0, u_*[C] = A\\ 
 & u(z_{C, i}) \in W^u(x_i), i = 0, \ldots, p; u(z_{C, \infty}) \in W^s(x_\infty)
\end{array}\right\} 
\end{equation}
and let $\overline{\mathcal{M}}_A(C, x_0, \ldots, x_p; x_\infty)$ be its Gromov compactification. Writing $f_A = \theta^{f_A^\theta}t^{f_A^t}f'_A$ with $f'_A \in H^*(M, \F)$ we thus have that for some $x \in Crit(f)$ whose coefficient for a chain level representative of $f'_A$ is nonzero, we have that
\begin{equation}
\overline{\M_A(\Delta_{f_A^\theta + 2f_A^t} \times C , p_-, p_+, \ldots, p_+, x)} = \Delta_{f_A^\theta + 2f_A^t} \times \overline{\mathcal{M}}_A(C, p_-,p_+ \ldots, p_+; x)
\end{equation}
is nonempty, where the left hand side is defined using using the perturbation data $(J^\infty_{z, w}, 0)$. Thus in particular we have that the Gromov compactification $\overline{\mathcal{M}}_A(C'; p_+, \ldots, p_+)$ of 
\begin{equation}
  \mathcal{M}_A(C'; p_+, \ldots, p_+) = \left\{\begin{array}{c|l}
u: C \to M &  \bar{\partial}_{J^\infty}u =0, u_*[C] = A\\ 
 & u(z_{C, i}) = p_+, i = 1, \ldots, p;
\end{array}\right\} 
\end{equation}
(where $C'$ of course refers to $C$ with $z_{C, 0}$ and $z_{C, \infty}$ forgotten) is nonempty. 
\begin{remark}
\label{remark:compactness-uniruled}
This standard argument implies that for arbitrary $J^\infty$, there is a $J^\infty$-holomorphic sphere through any $p^+ \in M$ if $M$ is $\F_p$-uniruled for some $p$. This is explained and first stated in \cite[Remark 2]{S-PRQS}.
\end{remark}
We will argue that this moduli space must be empty if we choose $J^\infty$ generically. 

Each component of the moduli space corresponds to a tree with a root vertex; the root vertex can be identified with $C$, and each vertex corresponds to a pseudoholomorphic sphere for $J^\infty$,  Moreover, if we collapse all non-root spheres then the points $z_{C, 1}, \ldots, z_{C, p}$ must map to corresponding points on $C$. In particular these points lie in distinct tree components. 

We will further stratify $\overline{\mathcal{M}}_A(C'; p_+, \ldots, p_+)$ by additionally considering how many of the marked points $z_{C, 1}, \ldots, z_{C, p}$ coincide upon replacing the root component with the $J^\infty$-holomorphic sphere that it may cover. The evaluation maps to $M$ used to define the intersection of the corresponding space of curves with the stable and unstable manifolds factor through smooth manifolds of maps in which we replace each sphere in the tree with the simple sphere that it covers; we will argue that the dimension of this manifold is too low for each stratum of $\overline{\mathcal{M}}_A(C'; p_+, \ldots, p_+)$ to be nonempty. 

On a semi-positive symplectic manifold, the Chern number of any $J^\infty$-holomorphic sphere is non-negative. Since the total Chern number of the tree is at most $1$, at most one sphere is Chern number $1$ (and is thus a simple sphere) while the remaining spheres are Chern number zero. In particular if there is a nonempty stratum then $c_1(A) \geq 0$. 
\begin{remark}
\label{rk:no-negative-Chern-numbers}
The same method of argument shows that for any $a, b \in H^*(M, \F)$, if $Q\Sigma_a(b) = f_0 + \sum_{A \neq 0} f_A q^A$ then $f_A \neq 0$ implies $c_1(A) \geq 0$. 
\end{remark}

We note that adding spheres of Chern number zero to the tree always decreases the dimension of the associated manifold. If the Chern number of the root component is $1$ then it is simple; in this case it suffices to upper bound the dimension of the configuration where the only sphere is the root component, which has dimension $2n+2-p(2n) < 0$ (unless $n=1$, in which case the computation of \cite{Wilkins} for $M=S^2$ together with Remark \ref{remark:compactness-uniruled} makes it possible to verify the theorem directly).

If the root component is not simple then its Chern number is zero and it also covers a sphere of zero Chern number, which be simple or constant. Whether it is simple or constant turns out not to matter for the dimension computation. Suppose the total Chern number is $1$. We reduce the case where the Chern number of the root component is zero to the case where either the sphere of Chern number $1$ is in a tree component that contains none of the $z_{C, i}$, or that it is; and there are no other holomorphic spheres. We then divide the cases into the number of distinct points there are after taking the images of the $z_{C, i}$ by the covering map; the worst case is manifestly when all the points are collapsed to one. So there are only two cases to check and they both have negative expected dimension. 

If the total Chern number is $0$ then the same reasoning applies but we have one more case to check: namely, the case when there are no components except the root component, which is a cover of a Chern number zero sphere on which all the marked points $z_{C, i}$ have been sent to a single point. In this case it is crucial that we have taken $J^\infty$ to be domain independent; the expected dimension of this configuration is
$2n-2n-4 <0$, where the $4$ comes from the automorphisms of the final curve that fix the single marked point. 
\end{proof}

Next, we prove the final technical proposition claimed in Section \ref{sec:steenrod-coefficients-must-be-calabi-yau}:
\begin{proof}[Proof of Proposition \ref{prop:bad-case-2}]
We argue exactly as above. The argument using Gromov compactness shows that the Gromov compactification $\overline{\M}_A(C''; p_+, \ldots, p_+; x)$ of 
\begin{equation}
      \mathcal{M}_A(C''; p_+, \ldots, p_+) = \left\{\begin{array}{c|l}
u: C \to M &  \bar{\partial}_{J^\infty}u =0, u_*[C] = A\\ 
 & u(z_{C, i}) = p_+, i = 1, \ldots, p; u(z_{C, \infty}) \in W^s(x)
\end{array}\right\} 
\end{equation}
(where $C''$ refers to $C'$ with $z_{C, 0}$ and $z_{C, \infty}$ forgotten) is nonempty for some $x \in Crit(f)$ with nonzero coefficient in a chain-level representative of $f'_i$. Because $\overline{W^s}(x)$ is a manifold with corners of dimension equal to that of $W^s(x)$, for the purpose of the dimension computation we can assume that the codimension constraint on $z_{C, \infty}$ is as if $z_{C, \infty}$ was constrained to lie strictly in $W^s(x)$, since the corresponding moduli spaces where $z_{C, \infty}$ maps to lower-dimensional strata of $\overline{W^s}(x)$ are only easier to rule out. We note also that $W^s(x) \cap \{p_+\} = \emptyset$.

If the root component is simple and has Chern number $2$ then it suffices to check the case where there are no additional Chern number $0$ spheres attached, and this case is ruled out by a dimension count. If the root component is a point then either there are two Chern number $1$ spheres, in which case we are in a moduli space of dimension strictly lower than that where the root component has Chern number $1$; or there is one Chern number two sphere, and we can reduce to the case where are no other spheres. The Chern number $2$ sphere cannot be attached to an input since $W^s(x) \cap \{p_+\} = \emptyset$; so it must be attached to the output. In this case, since $W^s(x) \cap \{p_+\} = \emptyset$, even if the sphere is a multiple cover, the covering map must still send the two marked points to distinct points, and so the dimension computation shows that this configuration cannot exist as $10-2-3-6<0$. If the root component has Chern number $1$, then all other spheres are Chern number zero or simple, and for the purpose of the dimension count we can assume that there is only one more sphere of Chern number $1$ (which is then also simple); there are two cases and each are ruled out by counting dimensions. Finally, if the root component is a double cover then all other spheres have Chern number zero, so we can assume there are no other spheres; as before the double cover can collapse the two inputs but cannot collapse an input and an output, and the dimension count $10-2-3-6<0$ excludes this configuration as well. 
\end{proof}

Finally, we prove, the remaining proposition used in Section \ref{sec:steenrod-coefficients-must-be-calabi-yau}:
\begin{proof}[Proof of Proposition \ref{prop:bad-case}]
We have $p=2$ fixed.
The computation is similar to the previous one: 
we will choose $(J^i_{z, w}, \eta^i_{z, w})$ converging to some $(J^\infty_{z, w}= J^\infty, 0)$ and analyze the resulting strata of the limit. However, to exclude one of the possible strata we will have to choose a particular Morse function to calculate cohomology, and a more careful constraint on $J^\infty$ and on the $J^i_{z, w}$ and $\eta^i_{z, w}$. We will have $x \in Crit(f)$ have nonzero coefficient in a chain-level representative of $f'_i$, as before.

Specifically, by \cite[II.27]{thom}, the Poincare dual homology class to $a$ can be represented by an oriented $2$-dimensional closed submanifold $N$ of $M$; by attaching a local handle to $N$ we can assume that the genus of $N$ is positive. By the $h$-principle it can be represented by a symplectic submanifold \cite{hprinciple}. We choose the Morse function $f$ such that there is a single critical point $y$ representing $a$ in integral cohomology, such that the closure of $W^u(y)$ is exactly $N$ and such that $N \cap \overline{W^s(x)}$ has dimension $1$: this can be done by choosing a Morse function on $N$ with a single maximum, extending by a negative definite quadratic form on a tubular neighborhood $U$ of $N$ with closure $\bar{U}$, and then extending to a Morse function on all of $M$. We require that $J^\infty|_{\bar{U}}$ is integrable,  with $N$ a complex submanifold; the projection from $U$ to $N$ holomorphic, making $\bar{U}$ into a closure of a neighborhood of zero of a holomorphic vector bundle structure on the normal bundle to $N$. We additionally fix the germ of a complex hypersurface $D$ containing $N$. Finally, we require that the $J^i_{z, w}$ agree with $J^\infty$ on $\bar{U}$ for all $w$ and $i$, and that $\eta^{i}_{z, w}$ is zero on $U$ for all $w$ and $i$. This constraint does not violate transversality because \emph{there are no $J^\infty$ holomorphic spheres contained entirely in $\bar{U}$}. Indeed, if there there was such a sphere, its projection to $N$ would give a holomorphic map to a surface of positive genus, which would have to factor through a point of $N$; so the curve would lie in a fiber of the projection  to $N$, which would define a bounded holomorphic map from a sphere to a $\C^2$, contradicting Liouville's theorem. 

Finally, we require that $J^\infty$ is chosen such that the \emph{moduli space of simple holomorphic spheres with two marked points, one passing through $N$ and one passing through a stratum of $\overline{W^s(x)}$, and having contact of order $2$ to $D$} is of expected dimension. The notion of spheres having contract of positive order to a holomorpic divisor was introduced in \cite{cieliebak-mohnke-1}, \cite{cieliebak-mohnke-2}, and an elaborate theory for the case where the divisor is locally defined near a point was established in \cite{mcduff-siegel-tangency}. The transversality results \cite[Prop. 6.9]{cieliebak-mohnke-1} and \cite[Prop 3.1]{cieliebak-mohnke-2} (see also \cite[Lemma 2.2.1]{mcduff-siegel-tangency}) suffice to show that this moduli space is of expected dimension for generic $J^\infty$; the dimension of this moduli space is $2$ less than the dimension of the corresponding moduli space of simple spheres with no order of contact  contraint at $D$.

As in the proof of Proposition \ref{prop:bad-case-2} we will assume that the constraint on the output is a codimension $6-|f'_i| = \codim W^s(x) = \codim\overline{W^s(x)} \geq 3$ constraint since $\overline{W^s(x)}$ is a manifold with corners compactifying $W^s(x)$.

If the root component has Chern number $2$ then all other components have Chern number zero, and so we can assume the only component is the root componnet for the purpose of upper bounding the dimension. If the root component is simple then the moduli space has dimension
$6+4-4-4-(6-|f'_i|)<0$. If the root component is a two-fold cover then the resulting curve is simple; as before we can reduce to the case where there are no other spheres.

If the images of $z_{C, 1}$ and $z_{C, 2}$ are distinct the dimension is only lower than in the previous case. If the images of $z_{C, 1}$ and $z_{C, 2}$ agree then we upper bound the dimension by $6+2-2-(6-|f'_i|)-4)<0$. 

If the Chern number of the root component is $1$ then it is simple. We can reduce to the case where there is exactly one other sphere of Chern number $1$, attahed to the root component either away from the marked points or at one of the marked points; all three cases have negative expected dimension.

Suppose the root component is a constant map. We reduce to two cases: either there is one additional sphere of Chern number $2$ and no other spheres, or there there are two additional spheres attached at different points. If there is a sphere of Chern number $2$ then either it is attached to an input, in which case if it is simple then lies in a moduli space of dimension bounded by $6+4-2-4-5$ (where we have used that one of the marked points on this sphere lies in the \emph{intersection} of the stable manifold of the output with the unstable manifold of the input). Alternatively, it is attached to the output; but then by \cite[Lemma 7.2b]{cieliebak-mohnke-1} this sphere must have contact of order $2$ to the divisor $D$ containing $N$, and thus lies in expected dimension $2$ lower than if it was only required to be contained in $N$; for a usual transverse intersection the expected dimension (if this sphere is simple) is $6+4-2-(6-|f'_i|)-4 \leq 1$, and so with the tangency constraint we have a negative upper bound. If this sphere is a double-cover then it might collapse the $N$ constraint and the $W^{s}(x)$ constraint, but then the unique marked point lies in $N \cap W^{s}(x)$ which has dimension $1$ and thus the expected dimension of this configuration is negative as well. This latter argument also resolves the case where the Chern number $2$ sphere was attached to an input and was a double cover.

Finally, suppose the Chern number of the root component is zero. Then we can again reduce to the case where there is additionally just one Chern number $2$ sphere or $2$ Chern number $1$ spheres. 

If the simplification of the root component sends $z_{C, 1}$ and $z_{C, 2}$ to distinct points just the root component together with its intersections with the two input cycles cannot exist for dimension reasons.

If the simplification of the root component sends $z_{C, 1}$ and $z_{C, 2}$ to the same point then: in the case with one more sphere of Chern number $2$ we attach to the input or the output, and get dimension at most $-1$ (these are degenerations of previously studied configurations) and if we have two Chern number $1$ spheres then we distribute them among input and output in the $3$ different ways possible and upper bound dimensions separately. 
\end{proof}

\printbibliography

\end{document}